\documentclass[11pt,a4paper]{amsart}
\usepackage{amsmath,amssymb,amsthm,amsfonts,enumerate,mathtools,pinlabel,float,makecell}
\usepackage[myheadings]{fullpage}
\usepackage{tikz-cd,tikz,pgf}
\usepackage{tikz-cd}
\usetikzlibrary{arrows.meta, bending}
\tikzset{C/.style={circle, minimum size=8mm,
                   node contents={},
                   append after command={\pgfextra{%
        \draw[-{Straight Barb[flex']}](\tikzlastnode.150) arc (150:450:4mm);}
                }}
        }
\usepackage{hyperref}
\hypersetup{
    colorlinks=true,  
    linkcolor=blue,
    citecolor=blue,
}
\usepackage{subcaption}
\newtheorem{cor*}{Corollary}
\newtheorem{thm*}{Theorem}
\newtheorem{lem*}{Lemma}
\newtheorem{prop*}{Proposition}
\newtheorem{qstn*}{Question}
\newtheorem{eqn*}{Equation}
\newtheorem{theorem}{Theorem}[section]
\newtheorem{cor}{Corollary}[theorem]

\newtheorem{lemma}[theorem]{Lemma}

\newtheorem{algo}[theorem]{Algorithm}

\theoremstyle{definition}
\newtheorem{defn}[theorem]{Definition}
\newtheorem{exmp}[theorem]{Example}

\newcommand{\Z}{\mathbb{Z}}

\newcommand{\orb}{\mathcal{O}}

\newcommand{\Mod}{\mathrm{Mod}}
\newcommand{\Teich}{\mathrm{Teich}}
\newcommand{\T}{Teichm\"uller }

\everymath{\displaystyle}

\makeatletter
\@namedef{subjclassname@2022}{\textup{2022} Mathematics Subject Classification}
\makeatother

\begin{document}

\title[Coordinates of the branch loci of cyclic actions]{Fenchel-Nielsen coordinates of the branch loci \\of cyclic actions}
\author{Atreyee Bhattacharya}
\address{(A. Bhattacharya) Department of Mathematics\\
Indian Institute of Science Education and Research Bhopal\\
Bhopal Bypass Road, Bhauri \\
Bhopal 462 066, Madhya Pradesh\\
India}
\email{atreyee@iiserb.ac.in}
\urladdr{https://sites.google.com/iiserb.ac.in/homepage-atreyee-bhattacharya/home?authuser=1}

\author{Satyajit Maity}
\address{(S. Maity) Department of Mathematics\\
Indian Institute of Science Education and Research Bhopal\\
Bhopal Bypass Road, Bhauri \\
Bhopal 462 066, Madhya Pradesh\\
India}
\email{msatyajit.194@gmail.com}

\author{Kashyap Rajeevsarathy}
\address{(K. Rajeevsarathy) Department of Mathematics\\
Indian Institute of Science Education and Research Bhopal\\
Bhopal Bypass Road, Bhauri \\
Bhopal 462 066, Madhya Pradesh\\
India}
\email{kashyap@iiserb.ac.in}
\urladdr{https://home.iiserb.ac.in/$_{\widetilde{\phantom{n}}}$kashyap/}



\begin{abstract} Let $S_g$ be a closed, connected, and oriented smooth surface of genus $g\geq 2$. Let the mapping class group of $S_g$ be denoted by $\mathrm{Mod}(S_g)$ and the Teichm\"{u}ller space of $S_g$ by $\mathrm{Teich}(S_g)$. It is known that $\mathrm{Mod}(S_g)$ acts by isometries on $\mathrm{Teich}(S_g)$ with respect to the Weil-Petersson metric. In this paper, we develop algorithms to describe the Fenchel-Nielsen coordinates of fixed points of the actions of certain finite cyclic subgroups of $\mathrm{Mod}(S_g)$ on $\mathrm{Teich}(S_g)$. As applications of these algorithms, we compute the Fenchel-Nielsen coordinates of the fixed points of three cyclic subgroups of orders $10$, $8$, and $4$, on $\mathrm{Mod}(S_2)$. 
\end{abstract}

\maketitle

\section{Introduction}
\label{sec:intro}
Let $S_g$ be the closed orientable surface of genus $g \geq 2$, and let $\Mod(S_g)$ be the mapping class group of $S_g$. Kerckhoff’s solution to the Nielsen realization problem \cite{Kerckhoff} established that any finite subgroup of $\Mod(S_g)$ can be realized as a group of isometries for some hyperbolic metric on $S_g$. More recently, in \cite{Realization1}, Parsad et al. developed methods for constructing explicit hyperbolic metrics that realize finite cyclic subgroups of $\Mod(S_g)$. In a subsequent work \cite{Realization2}, the authors gave explicit parametrizations of the fixed point sets of these cyclic actions as totally geodesic K\"ahler submanifolds of $\Teich(S_g)$ with respect to the Weil-Petersson metric.

It is a classical fact that $\Mod(S_g)$ acts properly discontinuously on the \T space by isometries with respect to both the \T and the Weil-Petersson metrics. However, barring surface rotations, explicit descriptions of these isometries in terms of the Fenchel-Nielsen coordinates of \T spaces via suitably chosen pants decompositions of the surface, is not well understood. Another natural question that follows in this context is how one can derive the explicit Fenchel-Nielsen coordinates of the \T classes represented by the hyperbolic structures realizing a given cyclic subgroup $H =\langle F\rangle$ of $\Mod(S_g)$. One possible approach to this end could be first to understand the isometry $F_\#$ induced by $F$ on $\Teich(S_g)$ and then compute its fixed points. This method is particularly effective when $F$ is represented by a surface rotation, as in that case, $F_\#$ is simply a permutation of the Fenchel-Nielsen coordinates of $\Teich(S_g)$. However, this method becomes quite challenging for an arbitrary periodic \( F \) for two main reasons: (a) it is difficult to derive \( F_\# \) in general, as an explicit factorization of \( F \) into Dehn twists may not be known; and (b) even if \( F_\# \) is known, solving the equation \( F_\#(x) = x \) in \( \Teich(S_g) \) is expected to be highly non-trivial. 
 
Let $F \in \Mod(S_g)$ be an irreducible periodic mapping class whose Nielsen representative has at least one fixed point on $S_g$ (also known as a \textit{irreducible Type 1 action}).  It was shown in \cite{Realization2} and \cite{Realization1} that in such a case the induced map $F_\#$ on $\Teich(S_g)$ has a unique fixed point, which is represented by a canonical semi-regular hyperbolic polygon $\mathcal{P}_{F}$ with a side-pairing whose rotation realizes $F$.  In this paper, using the results from \cite{Realization2} and \cite{Realization1}, we provide an algorithm  to compute the Fenchel-Nielsen coordinates of the fixed point of $F_\#$ that bypasses the need to derive a closed-form expression for $F_\#$. Given a Type 1 irreducible action $F$ on $S_g$ as above,  the following algorithm provides a description of the fixed point of its induced action on $\Teich(S_g)$ as summarised below:
\begin{algo}\label{algo:1}
\begin{enumerate}[\textit{Step} 1.]
\item Start with the unique hyperbolic polygon $\mathcal{P}_{F}$ which realizes $F$ as an isometry, as mentioned in Theorem 1 in \cite{Realization1} and Proposition 4.1 in \cite{Realization2}.
\item  Construct a suitable pants decomposition as follows:
\begin{enumerate}[\textit{Step} 2(a).]
\item Choose a homotopically non-trivial simple closed curve, possibly along the boundary of the polygon, and consider its orbit under the action.
\item If the number of homotopically disjoint curves in that orbit is $3g -3$, we consider this collection as our desired pants decomposition.
\item Otherwise, we consider a suitable homotopically non-trivial simple closed curve in the complement of the previous orbit and consider its orbit under the action. 
\item Repeat steps 2(a) - 2(c) until the total number of homotopically disjoint simple closed curves from distinct orbits adds up to $3g-3$.  
\end{enumerate}
\item  Viewing the polygon $\mathcal{P}_{F}$ as a hyperbolic polygon in the Poincar\'e disc $\mathbb{D}$ with proper identifications, compute the length and twist parameters associated with the corresponding pants curves and thereby obtain the Fenchel-Nielsen coordinates of the hyperbolic structure given by $\mathcal{P}_{F}$ using basic hyperbolic trigonometry. 
\end{enumerate}
\end{algo} 
As a concrete demonstration of the algorithm, we obtain the Fenchel-Nielsen coordinates of the Type 1 irreducible actions of orders $10$ and $8$ on $S_2$ (see Section \ref{sec. 3} for details). These are known~\cite{Harvey1966,AW} to be the cyclic actions of largest possible orders on $S_2$. The description of the fixed points of the corresponding Type 1 irreducible actions is as follows.
\begin{cor*} \label{Cor:order 10}
Let $F \in \Mod(S_2)$ be a Type 1 irreducible action and let $(c_1, c_2, c_3, t_1, t_2, t_3)\in \Teich(S_2)$ be the Fenchel-Nielsen coordinates of the  unique fixed point of the induced $\langle F_\# \rangle$-action in $\Teich(S_2)$. 
\begin{enumerate}[(i)]
\item If is realized as a $2\pi/10$ rotation of the canonical regular hyperbolic $10$-gon $\mathcal{P}_F$ with opposite sides identified, then
\begin{eqnarray*}
c_1 &=& c_2=2\operatorname{arcosh}\left(\frac{2+\sqrt{5}}{2}\right),\ \ \ 
c_3 = 2\operatorname{arcsinh}\left(\sqrt{\frac{5+3\sqrt{5}}{2}}\right),\\
t_1&=& t_2= 2\operatorname{arcosh}\left(\frac{1}{4}\sqrt{25+9\sqrt{5}}\right) 
\ \text{ and} \ \ t_3=-\operatorname{arcosh}\left(\frac{3+\sqrt{5}}{2}\right).\\
\end{eqnarray*}
\item If is realized as a $2\pi/8$ rotation of the canonical regular hyperbolic $8$-gon $\mathcal{P}_F$ with opposite sides identified, then
\begin{eqnarray*}
&& c_1=c_2=t_1=t_2=2\operatorname{arcosh}\left(1+\sqrt{2}\right),\\ 
&&c_3=2\operatorname{arsinh}\left(2 \sqrt{4+3\sqrt{2}} \right) \ \text{ and }t_3=-2\operatorname{arcosh}\left(\sqrt{2+\sqrt{2}}\right).
\end{eqnarray*}
\end{enumerate}
\end{cor*} 

\noindent Algorithm~\ref{algo:1} also helps in deriving fixed points of Type 1 irreducible actions that are realized as higher powers of Type 1 irreducible actions on $S_g$ ($g\geq 2$). For instance, consider the Type 1 irreducible action of order $5$ on $S_2$ that can be realized as $F^2$, where $F$ is the Type 1 irreducible action of order $10$ from Corollary \ref{Cor:order 10}. Then the induced actions of $F$ and $F^2$ on $\Teich(S_2)$ share the same fixed point. 

It turns out that the above algorithm can be extended to describe coordinates of fixed points of certain periodic reducible actions, also called \textit{Type 2 actions} on $S_g$ (see Algorithm~\ref{algo:compatible_pair}). In \cite{Realization1}, it was shown that given irreducible Type 1 actions $F$ and $F^{-1}$ on $S_{g}$ realized by isometries of the same hyperbolic structure $\mathcal{P}_F$,  a periodic Type 2 action $G$ on $S_{2g}$ can be realized by an hyperbolic structure obtained by removing two identical invariant disks around the centers of two copies of $\mathcal{P}_F$ and then identifying the resultant boundary components.  For $i=1,2$,  such a pair $F_i$ of actions on $S_{g_i}$ where fixed points at the centers of the polygons $\mathcal{P}_{F_i}$ induce local rotational angles $\theta_i$ satisfying $\theta_1+\theta_2 \equiv 0\pmod{2 \pi}$, are said to form a \textit{compatible pair }$(F_1,F_2)$.  Using this construction, we obtain an algorithm to describe the explicit Fenchel-Nielsen coordinates of fixed point sets of the isometric actions in $\Teich(S_{2g})$ that correspond to the aforementioned Type 2 actions on $S_{2g}$ (see Section \ref{sec.4} for details). 

\begin{algo}\label{algo:intro_compatible_pair} Given a Type 1 irreducible $F \in \operatorname{Mod}\left(S_g\right)$ of order $n$, we consider the compatible pair $G:=(F, F^{-1})$ realized via compatibility along a common fixed point. Then an algorithm to describe the Fenchel-Nielsen coordinates of the fixed points of the induced $\langle G_{\#}\rangle$ action on $\Teich(S_{2 g})$ is outlined below:
\begin{enumerate}[\textit{Step} 1:]
\item Start with the unique hyperbolic polygon $\mathcal{P}_F$ realizing $F$ as an isometry as mentioned in Theorem 1 in \cite{Realization1} and Proposition 4.1 in \cite{Realization2}.
\item If $\mathcal{P}_F$ has $k$ sides, obtain a hyperbolic $2k$-gon $\mathcal{P}_G$ representing a hyperbolic structure on $S_{2g}$ that realizes $G$ as an isometry, as follows:
\begin{enumerate}[\textit{Step} 2(a)]
\item Construct a hyperbolic $(k+1)$-gon $\mathcal{P}^{\prime}$, with $k$ sides of equal length such that the corresponding $(k-1)$ interior angles bounded by those $k$ sides are equal. The remaining side, distinguished as $\ell$, is not necessarily equal to the $k$ sides mentioned earlier. 
\item In order that $\mathcal{P}^{\prime}$ realizes a hyperbolic structure on $S_{g}^1$ with $F$ as an isometry, it follows from basic hyperbolic trigonometry that the remaining two interior angles of $\mathcal{P}^{\prime}$ must be equal to each other (not necessary equal to the $k-1$ interior angles mentioned above). 
\item On these $k$ equal sides apply the side-pairing relations as in the polygon $\mathcal{P}_F$. Note such a polygon $\mathcal{P}^{\prime}$ always exist representing $F$ as an isometric action on the hyperbolic surface $S_{g}^1$.
\item Construct another hyperbolic ($k+1$)-gon $\mathcal{P}^{\prime \prime}$ by reflecting $\mathcal{P}^{\prime}$ along its distinguished side $\ell$. $\mathcal{P}^{\prime \prime}$ then has the same side-pairing relations as of $\mathcal{P}^{\prime}$.
\item Obtain the desired hyperbolic $2k$-gon $\mathcal{P}_G$ by combining $\mathcal{P}^{\prime}$ and $\mathcal{P}^{\prime \prime}$ identified along the common distinguished side $\ell$. Note that by construction, all sides of $\mathcal{P}_G$ has equal length, $(2k-2)$ interior angles are equal and the remaining two interior angles are also equal to each other (but not necessarily to the $(2k-2)$ interior angles).  
\end{enumerate} 
\item Construct a suitable pants decomposition on $\mathcal{P}_G$ in the following way:
\begin{enumerate}[\textit{Step} 3(a)]
\item Pick the common distinguished side of $\mathcal{P}^{\prime}$ and $\mathcal{P}^{\prime \prime}$ as the first pants curve. Call it $\gamma$.
\item  Choose a homotopically non-trivial simple closed curve, possibly along the boundary of the polygon $\mathcal{P}^{\prime}$, disjoint from $\gamma$ up to homotopy, and consider its orbit under the action.
\item If the number of homotopically disjoint curves in that orbit is $3 g-2$, consider their counterparts from $\mathcal{P}^{\prime \prime}$. These $6g-4$ curves together with $\gamma$, form a pants decomposition for $S_{2 g}$.
\item Otherwise, find a suitable nontrivial simple closed curve in $\mathcal{P}^{\prime}$, lying in the complement of the previous orbit and $\gamma$ (such a curve always exists), and consider its orbit under the action.
\item Repeat the steps 3(b) - 3(d) until the total number of homotopically disjoint simple closed curves obtained from distinct orbits in the polygon $\mathcal{P}^{\prime}$ (other than $\gamma$) adds up to $3 g-2$.
\end{enumerate}
\item Viewing $\mathcal{P}_G$ inside the Poincaré disc with proper identifications, compute the length and twist parameters associated with the corresponding chosen pants curves thereby obtaining the Fenchel-Nielsen coordinates of the hyperbolic structure with the help of hyperbolic trigonometry.
\end{enumerate}
\end{algo}

As an application of Algorithm~\ref{algo:intro_compatible_pair}, we have the following corollary.
\begin{cor}\label{Cor:Type2}
Consider the periodic mapping class $F \in \Mod(S_1)$ of order $4$ realized as the $\pi/2$ rotation of a square with its opposite sides identified and its inverse $F^{-1}$. Let $G \in \Mod(S_2)$ be the Type 2 action of order $4$ realized as a compatible pair $(F,F^{-1})$. Then the fixed points of the induced $\langle G_{\#} \rangle$-action in $\Teich(S_2)  \approx \mathbb{R}^6$ is a two dimensional submanifold of $\Teich(S_2)$ whose Fenchel-Nielsen coordinates are given by:
$$\left\{\left(\gamma_1, \gamma_2, \gamma_1, t, 0,-t\right)\right\},$$ where $\gamma_1=2 \operatorname{arcosh}\left(\cosh \frac{s}{2} \sin \alpha\right), \gamma_2=\operatorname{arcosh}\left(\cosh ^4 s-2 \cosh ^3 s+2 \cosh s\right)$,
$t= \frac{\gamma_1}{2}-\operatorname{arcoth}(\sinh (s-x) \tan \alpha), x=\operatorname{artanh}\left(\operatorname{coth} \frac{s}{2} \cos \alpha\right)$, with $s>\operatorname{arcosh}\left(\cot ^2 \frac{\alpha}{2}\right)$, and $0< \alpha<\frac{\pi}{3}$.
\end{cor}

Finally, in Section~\ref{sec:conclusion}, when $F$ is an irreducible Type 1 action, we sketch an algorithm that helps describe the corresponding induced isometric $\langle F_\# \rangle$-action on the \T space (see Algorithm~\ref{algo:induced action}). We also provide a supporting example for this algorithm (see Example \ref{exmp:induced action}). As our focus in this article is primarily on the description of fixed points, induced actions on \T space will be studied more rigorously in upcoming works. 
\section{\textbf{Preliminaries}}\label{sec. 2}
In this section we recall certain basic concepts and results about finite group actions on surfaces that are relevant to this article and also introduce the notations to be followed.

\subsection{Cyclic actions of finite order on surfaces}
Let $F \in \Mod(S_g)$ be a periodic mapping class of order $n$. The \textit{Nielsen Realization} theorem \cite{Kerckhoff, Neilsen1932},  $F$ is represented by an order-$n$ isometry $\mathcal{F}$ of some hyperbolic metric on $S_g$.  Such an $\mathcal{F}$ is called a \textit{Nielsen representative} of $F$.  The $\langle\mathcal{F} \rangle$-action on $S_g$ induces a branched covering $S_g \longrightarrow S_g/{C_n}$, where $C_n = \langle \mathcal{F} \rangle$. The quotient orbifold $\orb_F:=S_g/{C_n}$, also known as the \textit{corresponding orbifold} of $F$, has \textit{signature} $\text{sig}(\orb_F)=(g_0; n_1, \dots, n_\ell)$. From Thurston's orbifold theory \cite{ThurstonGT3M}, the $C_n$-action induces the following short exact sequence 

$$1 \rightarrow \pi_1(S_g) \rightarrow \pi_1^{\text{orb}}(\orb_F) \xrightarrow{\rho} C_n \rightarrow 1,$$
\noindent
where $\pi_1^{\text{orb}}(\orb_F)$, called \textit{orbifold fundamental group},  is a \textit{Fuchsian group} \cite{Katok} with presentation
$$\langle \alpha_1, \beta_1, \dots, \alpha_{g_0}, \beta_{g_0}, \xi_1, \dots, \xi_\ell \, |\, \xi_1^{n_1}=\dots=\xi_\ell^{n_\ell}=1, \prod_{i=1}^\ell \xi_i=\prod_{j=1}^{g_0} [\alpha_i, \beta_i] \rangle.$$ 

The epimorphism $\rho: \pi_1^{\text{orb}}(\orb_F) \longrightarrow C_n $ also known as the \textit{surface-kernel map} \cite{Harvey1966}, is of the form $\rho(\xi_i)=t^{(n/n_i)c_i}$, where $C_n=\langle t \rangle$ and $c_i\in \Z_{n_i}^{\times}$ for $i=1, \dots, \ell$. From the geometric interpretation, it is known that the quotient orbifold $(\orb_F)$ has $\ell$ distinguished  \textit{cone points} of orders $n_1, \dots, n_\ell$ respectively. Each $x_i$ lifts under the branched covering map $S_g \longrightarrow S_g/{C_n}$ to an orbit of size $n/{n_i}$ on $S_g$, and the $\langle\mathcal{F}\rangle$-action induces a local rotation by an angle of $2\pi c_i^{-1}/{n_i}$ in a neighborhood around each points of this orbit. Therefore, we can associate a $C_n$ action on $S_g$ with the following data.
   
\begin{defn}
\label{defn:data-set}
A \textit{cyclic data set of degree $n$} is a tuple 

\begin{center}
    $D= (n, g_0, r; (c_1, n_1), (c_2, n_2), \dots, (c_l, n_l))$,
\end{center}

    where $n \geq 2, g_0 \geq 0, 0 \leq r \leq n-1$, and $l \geq 0$ are integers such that:
\begin{enumerate}

\item[(i)] $r > 0$ if and only if $l=0$, and when $r > 0$, we have gcd$(r, n) = 1$,
\item[(ii)] $2 \leq n_i \leq n$, and $n_i | n$ , for all $i$,
\item[(iii)] $1 \leq c_i \leq n_i - 1$, and gcd$(c_i, n_i)=1$, for all $i$,
\item[(iv)] for all $i$, lcm$(n_1, \dots, \widehat{n_i}, \dots, n_l) =$ lcm$(n_1, \dots, n_l)$, and if $g_0 = 0$, then lcm$(n_1, \dots, n_l) = n$, and 
\item[(v)] $\sum_{i=1}^{l} \frac{n}{n_i}c_i \equiv 0$ (mod $n$).
\end{enumerate}
\noindent
The number $g_0$ is the genus of the corresponding quotient orbifold and the number $g$ determined by the following Riemann-Hurwitz equation 

\begin{center}
    $\frac{2-2g}{n} = 2-2g_0 + \sum_{i=1}^{l} (\frac{1}{n_i} - 1)$
\end{center}

is called the \textit{genus} of the data set, denoted by $g(D)$.
\end{defn}

The quantity $r\geq 0$ in the data set $D$, denoted by $r(D)$, is positive if and only if $D$ represents a free rotation of $S_{g(D)}$ by an angle of $2\pi r(D)/n$, and if $r=0$, we omit writing it in the data set $D$. Also, if a pair $(c_i, n_i)$ occurs more than once, we use the notation $(c_i, n_i)^{[m_i]}$ to denote that the pair $(c_i, n_i)$ occurs with multiplicity $m_i$ in the data set $D$. 

The following theorem of \textit{Nielsen} \cite{Nielsen1937} together with the \textit{Nielsen Realization} theorem \cite{Kerckhoff, Neilsen1932} assert that data sets represent the periodic elements of $\Mod(S_g)$,  up to conjugacy.

\begin{theorem}
\label{thm:Nielsen}
For $g \geq 1$, there exists a bijective correspondence between the conjugacy classes of $C_n$-actions on $S_g$ and the cyclic data sets of degree $n$ and genus $g$.
\end{theorem}

\noindent In view of Theorem~\ref{thm:Nielsen}, we will denote the data set representing the conjugacy class of a cyclic action $F$ by $D_F$. 

\subsection{Classification of Periodic Actions}
Based on the above theorem, the periodic mapping classes are classified into the following three categories. 

\begin{defn}
For $g \geq 1$, let $F \in \Mod((S_g)$ be a periodic mapping class. Then $F$ is said to be a 
\begin{enumerate}
\item[(i)] \textit{Rotational action}, if the action is a rotation of the surface obtained as the restriction of a rotation of $\mathbb{R}^3$ under some fixed embedding $S_g \hookrightarrow \mathbb{R}^3$, or equivalently either $D_F=(n, g_0, r; -)$ with $r \neq 0$, $r \in \mathbb{Z}_n^{\times}$ or 
  \begin{center}
      $D_F= (n, g_0; \underbrace{(c, n), (n-c, n), ..., (c, n), (n-c, n)}_{k-\text{pairs}})$
  \end{center}
  for some integers $k \geq 1$ and $c \in \mathbb{Z}_n ^{\times}$, and $k=1$ if and only if $n \geq 2$. 

\item[(ii)] \textit{Type 1 action} if $D_F=(n, g_0; (c_1, n_1), (c_2, n_2), (c_3, n))$, where $g_0 \geq 0$, $n_i \mid n$, and $c_i \in \mathbb{Z}_{n_i}^{\times}$. 

\item[(iii)] \textit{Type 2 action} if $D_F$ is neither a rotational nor a Type 1 action.
\end{enumerate}
\end{defn}

Let us now define the \textit{reducible} and \textit{irreducible} actions and their characterization given by the following theorem of \textit{Gilman} \cite{Gilman}.

\begin{defn}
For $g \geq 1$,  a periodic mapping class in $\Mod(S_g)$ is called \textit{reducible} if it preserves a multicurve (called a \textit{reduction system}) in $S_g$; Otherwise, it is called \textit{irreducible}.
\end{defn}

\begin{theorem}
A periodic mapping class $F \in \Mod(S_g)$ is \textit{irreducible} if and only if the corresponding quotient orbifold $\orb_F$ is a sphere with three cone points.   
\end{theorem}

\subsection{Geometric realizations of Type $1$ Actions} The following theorem provides a geometric realization of irreducible Type 1 actions.
  \begin{theorem} [{\cite[Proposition 4.1]{Realization2}},{\cite[Theorem 2.7]{Realization1}}]\label{thm:geom_real}
For $g \geq 2$, let $F$ be an irreducible Type $1$ $C_n$-action on $S_g$ whose conjugacy class is represented by $D_F=(n, 0; (c_1, n_1), (c_2, n_2), (c_3, n))$. Then $F$ can be realized as a rotation by $2\pi c_3^{-1}/n$ around the center of a hyperbolic polygon $\mathcal{P}_F$ with $k(F)$ many sides of equal length, $\theta(F)$ angles at the vertices, and $W(\mathcal{P}_F)$ side-pairing relations where 
   \[ 
   k(F)=
   \begin{cases}
     2n, & \text{if } n_1, n_2 \neq 2,  \\
     n, & \text{otherwise},
   \end{cases}
   \]
   \[ 
   \theta(F)=
   \begin{cases}
     2\pi/n_1 \text{ and } 2\pi/n_2, & \text{if } n_1, n_2 \neq 2,  \\
     2\pi/n_2, & \text{if } n_1=2,
   \end{cases}
   \]
   and for $0 \leq m \leq n-1$,\\
   \[ 
   W(\mathcal{P}_F)=
   \begin{cases}
     \prod_{i=1}^{n} a_{2i-1}a_{2i} \text{    with } a_{2m+1}^{-1} \sim a_{2z}, & \text{if } k(F)=2n,  \\
     \prod_{i=1}^{n} a_i \text{    with } a_{m+1}^{-1} \sim a_{z}, & \text{if } k(F)=n, 
   \end{cases}
   \]
   where $z \equiv m+ qj$ (mod $n$), $q= n c_3^{-1}/n_2$, and $j= n_2 - c_2$.  Moreover, the polygon $\mathcal{P}_{F}$ describes the unique hyperbolic metric on $S_g$ realizing $F$ as an isometry.
  \end{theorem}
 
\begin{exmp}\label{ex:order10}
The unique realization of the irreducible Type $1$ action $F$ on $S_2$ (described in Corollary \ref{Cor:order 10} of Section \ref{sec:intro}) with $D_F=(10, 0; (1, 2), (2, 5), (1, 10))$ is shown in Figure \ref{Figure1} below. Here $F$ is realized as a $2 \pi / 10$ rotation about the center of a regular hyperbolic $10$-gon with interior angles equal to $2 \pi / 5$ and opposite sides identified. In this polygon, the center (which is the fixed point of the action) corresponds to the cone point $(1,10)$, the vertices $\{A, B\}$ correspond to the cone point $( 2,5)$ and the midpoints of the sides correspond to the cone point $(1,2)$.

\begin{figure}[h]
  \centering
    \includegraphics[width=40ex]{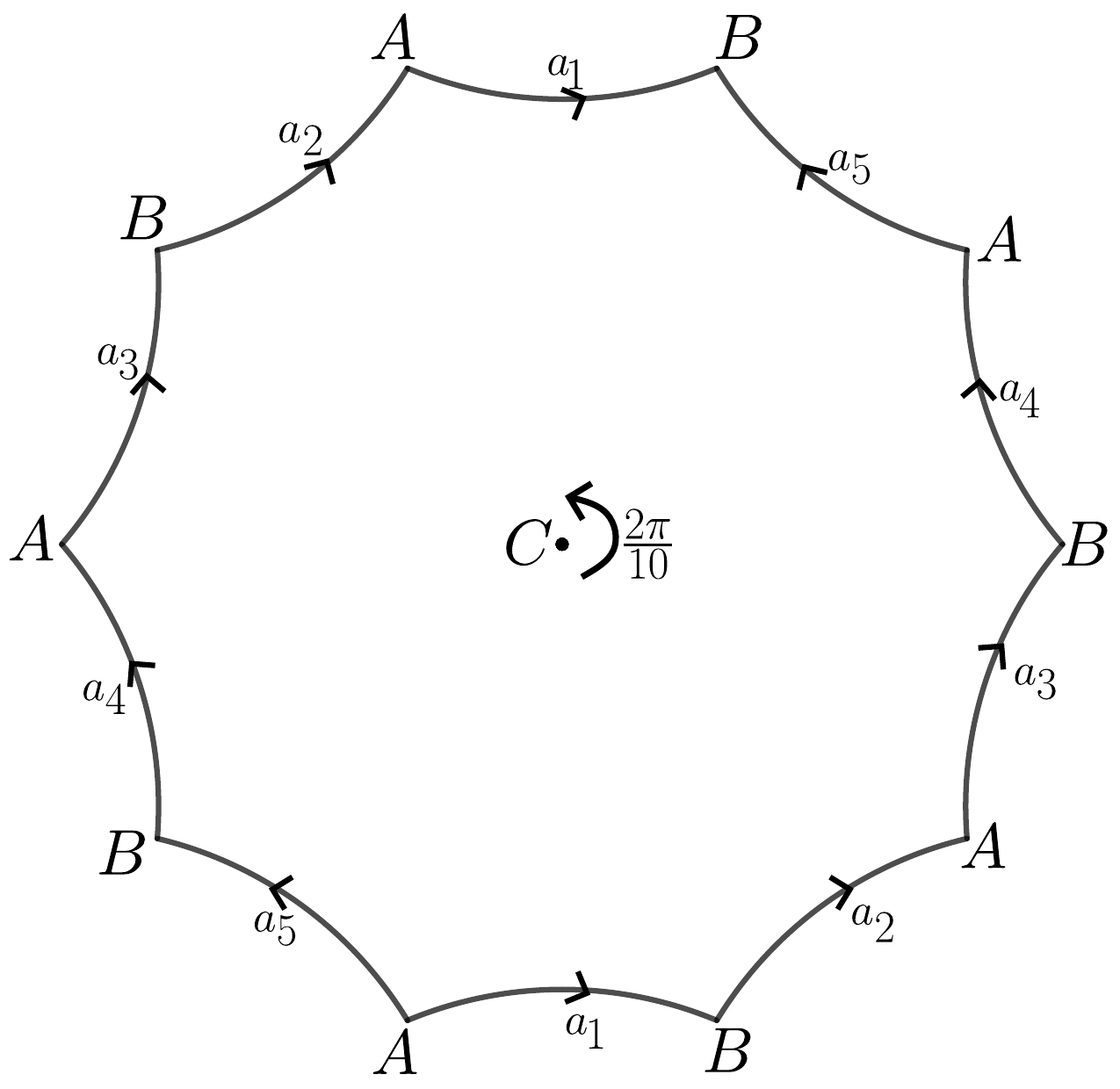}
    \caption{A realization of the irreducible Type 1 action $F$ on $S_2$.}
    \label{Figure1}
\end{figure}
\end{exmp}
\subsection{Compatibility of two cyclic actions} In this subsection, we describe a method of constructing a new reducible cyclic action by pasting a pair of irreducible Type 1 actions along compatible orbits where the induced rotations are equal. We refer the reader to~\cite{Realization1} for more details.
\begin{defn}
\label{defn:rs-comp}
For $i=1,2$, let $F_i  \in \Mod(S_{g_i})$ be periodic mapping classes with
$$D_{F_i}=(n, g_{0, i}; (c_{i, 1}, n_{i, 1}), (c_{i, 2}, n_{i, 2}), ..., (c_{i, l_i}, n_{i, l_i})).$$ 
Then $F_1$ and $F_2$ are said to form an \textit{$(r, s)$-compatible pair $(F_1, F_2)$} if there exists $1 \leq r \leq l_1$ and $1 \leq s \leq l_2$ such that 
\begin{enumerate}
      \item[(i)] $n_{1, r}=n_{2, s}=k$ and
      \item[(ii)] $c_{1, r} + c_{2, s} \equiv 0$ (mod $k$).
      \end{enumerate}
\end{defn}      
\noindent The following lemma is a direct consequence of Definition~\ref{defn:data-set}.
\begin{lemma}
An $(r, s)$-compatible pair $(F_1, F_2)$ of periodic mapping classes $F_i \in  \Mod(S_{g_i})$ as in Definition~\ref{defn:rs-comp} defines a reducible periodic mapping class $F \in \Mod(S_g)$,  where:
\begin{equation*}
\label{}
\begin{split}
   D_F :=&(n, g_{0, 1} + g_{0, 2}; (c_{1, 1}, n_{1, 1}), \dots, \widehat{(c_{1, r}, n_{1, r})}, \dots, (c_{1, l_1}, n_{1, l_1}),\\  
&(c_{2, 1}, n_{2, 1}), \dots, \widehat{(c_{2, s}, n_{2, s})}, \dots, (c_{2, l_2}, n_{2, l_2}))
\end{split}
\end{equation*}
and  $g = g(D_{F_1}) + g(D_{F_2}) + n/k -1$.
\end{lemma}

\begin{exmp}
Consider the irreducible Type 1 action $F \in \Mod(S_1)$ of order $4$ with $$D_F=(4, 0; (1, 2), (1, 4),(1,4)) \text{ and } D_{F^{-1}}=(4, 0; (1, 2), (3, 4),(3,4)).$$ Then $(F,F^{-1})$ forms a $(3, 3)$-compatible pair that defines a reducible periodic mapping class $G \in \Mod(S_2)$ of order $4$ with $D_G=(4, 0; (1, 2),(1,2), (1, 4), (3, 4))$ whose realization is shown in Figure \ref{Fig.: example of compatibility}.

\begin{figure}[h]
     \centering
        \includegraphics[width=45 ex]{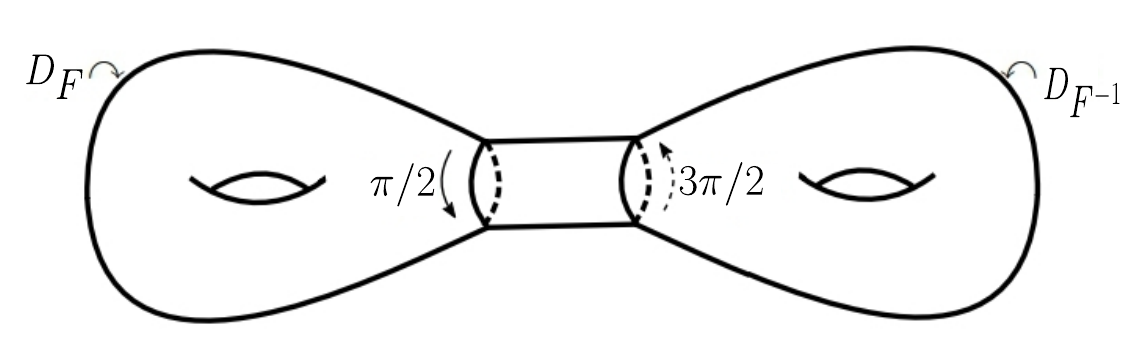}
         \caption{A realization of $G$ as a (3,3)-compatible pair.}
         \label{Fig.: example of compatibility}
\end{figure}

\end{exmp}
\subsection{The \T space of $S_g$} Let HypMet $\left(S_g\right)$ denote the set of all hyperbolic metrics on $S_g$ and $\operatorname{Diff}_0\left(S_g\right)$ denote the group of all diffeomorphisms of $S_g$ that are isotopic to identity. The \T space $\Teich(S_g)$ of $S_g$ is the quotient space $$
\operatorname{Teich}\left(S_g\right)=\operatorname{HypMet}\left(S_g\right) / \operatorname{Diff}_0\left(S_g\right)
$$
where $\operatorname{Diff}_0\left(S_g\right)\left(\operatorname{Diffeo}^{+}\left(S_g\right)\right)$ acts on $\operatorname{HypMet}\left(S_g\right)$ via pulling the metrics back, i.e.

$$
f \cdot \xi=f^*(\xi), \  \text{ for all } \ f \in \operatorname{Diff}_0\left(S_g\right) \text { and } \xi \in \operatorname{HypMet}\left(S_g\right) .
$$
The above action induces a natural action of $\Mod(S_g)$ on $\Teich(S_g)$ as follows: Given $F=[f] \in \operatorname{Mod}\left(S_g\right)$ and $[\xi] \in \operatorname{Teich}\left(S_g\right)$, we have $F \cdot[\xi]=\left[f^*(\xi)\right].$

Let $\gamma$ be an essential simple closed curve on $S_g$ ($g \geq 2$), $[\gamma]$ its homotopy class, and $[\xi] \in \Teich(S_g)$, where $\xi \in \operatorname{HypMet}\left(S_g\right)$. If $\gamma^{\prime} \in [\gamma]$ denotes the unique geodesic with respect to $\xi$ in $\left(S_g, \xi\right)$, then the length of $\gamma$ with respect to $[\xi]$, denoted by $\ell_{\gamma}([\xi])=\ell_{\gamma}(\xi)$, is defined as the $\xi$-length of $\gamma^{\prime}$. Let $\mathcal{P}=\left\{\gamma_1, \gamma_2, \ldots, \gamma_{3 g-3}\right\}$ be a pants decomposition of $S_g$. The above definition can be used to define the length parameters of $\mathcal{P}$ with respect to $[\xi]$ as the $(3g-3)$-tuple $(\ell_{\gamma_1}([\xi]), \ldots , \ell_{\gamma_{3g-3}}([\xi]))$ of lengths of the pants curves with respect to $[\xi]$. Let $\mathcal{S}=\left\{c_1, c_2, \ldots, c_n\right\}$ be a multicurve of $n$ disjoint simple closed curves on $S_g$, referred to as \textit{seams curves} (\cite[Chapter 10]{FM}) for the pants decomposition $\mathcal{P}$. Then the \textit{twist parameter of $\gamma_i$} with respect to $[\xi]$ is the signed distance (measuring along the unique $\xi$-geodesic representative $\gamma_i^{\prime}$ of $[\gamma_i]$) between the two points where the two common perpendiculars intersect $\gamma_i^{\prime}$ by following any one of the two seams curve passing through it (see \cite{FM, Bruno} for more details). We call a twist to be positive if one of the two common perpendiculars is on the left of the other (in the universal cover), and negative otherwise.

The following theorem due to Fenchel and Nielsen (\cite{Fenchel-Nielsen}) provides a parametrization of $\Teich(S_g)$ in terms of the Fenchel-Nielsen coordinates obtained via the length and twist parameters of curves in a pants decomposition of $S_g$.  
\begin{theorem}[\textbf{Fenchel-Nielsen}]\label{Thm:F-N}
Let $\mathcal{P}=\{\gamma_1, \gamma_2,\dots, \gamma_{3g-3}\}$ be a pants decomposition of $S_g$. Suppose $\ell_{\gamma_i}$ and $t_{\gamma_i}$ denotes the length and twist parameters of $\gamma_i$. Then 
 $$\Teich(S_g) \cong \mathbb{R}_+^{3g-3} \times \mathbb{R}^{3g-3} \cong \mathbb{R}^{6g-6}$$
  via the map 
  $$\chi=[\xi] \longmapsto (\ell_{\gamma_1}(\xi), \dots, \ell_{\gamma_{3g-3}}(\xi), t_{\gamma_1}(\xi), \dots, t_{\gamma_{3g-3}}(\xi)).$$
\end{theorem} 

\section{Algorithms to describe fixed points of cyclic actions on the \T space}
\label{sec. 3}
In this section, we first provide an algorithm to describe the Fenchel-Nielsen coordinates of the unique fixed points of the \T maps induced by Type 1 irreducible actions on $S_g$ ($g\geq 2$), realized as semi-regular hyperbolic polygons as described in Theorem \ref{thm:geom_real}, followed by a few important applications of the algorithm. This naturally extends to an algorithm to describe the Fenchel-Nielsen coordinates of the fixed points of the \T map induced by an $(r,s)$-compatible pair $(F,F^{-1})$ where $F$ is an irreducible Type 1 action of $S_g$ ($g\geq 1$). 

\subsection{Branch loci of irreducible Type 1 actions.}
\begin{algo}\label{algo:main}Given a Type $1$ irreducible $F \in \Mod(S_g)$, an algorithm to describe the fixed point of its induced action on $\Teich(S_g)$ is summarised as follows:
\begin{enumerate}[\textit{Step} 1:]
\item Start with the unique hyperbolic polygon $\mathcal{P}_{F}$ which realizes $F$ as an isometry, as mentioned in Theorem \ref{thm:geom_real}.
\item Construct a pants decomposition in the following way:
\begin{enumerate}[\textit{Step} 2(a):]
\item Choose a homotopically non-trivial simple closed curve, possibly along the boundary of the polygon, and consider its orbit under the action.
\item If the number of homotopically disjoint curves in that orbit is $3g -3$, we consider this collection as our desired pants decomposition.
\item Otherwise, find a homotopically non-trivial simple closed curve in the complement of the previous orbit and consider its orbit under the action (the existence of such a curve is guaranteed). 
\item Repeat steps 2(a) - 2(c) until the total number of homotopically disjoint simple closed curves from distinct orbits adds up to $3g-3$.  
\end{enumerate}
\item Viewing the polygon $\mathcal{P}_F$ as a hyperbolic polygon in the upper-half plane $\mathbb{H}^2$ or the Poincar\'e disc $\mathbb{D}$ with proper identifications, we can compute the length and twist parameters associated with the corresponding pants curves and thereby obtain the Fenchel-Nielsen coordinates of the hyperbolic structure given by $\mathcal{P}_F$ using basic hyperbolic trigonometry.
\end{enumerate}
\end{algo}
\subsection{Applications of Algorithm \ref{algo:main}.} To demonstrate several applications of Algorithm~\ref{algo:main}, we begin with the Type 1 irreducible mapping classes of the surface \( S_g \) with orders \( 4g +2 \) and \( 4g \). These actions are significant as they represent the cyclic actions of the largest possible orders in the mapping class group \( \Mod(S_g) \) \cite{Harvey1966,AW}. For clarity and conciseness, we will detail the case when \( g = 2 \). The general case for \( g \geq 3 \) can be derived using a similar approach; however, we have chosen not to include it in this manuscript due to its complexity.
\begin{cor} \label{cor:order 10}
 \label{cor:order10_fixedpt}

Let $F \in \Mod(S_2)$ be a Type 1 irreducible action with $D_F=(10, 0; (1, 2), (2, 5), (1, 10))$. The Fenchel-Nielsen
 coordinates of the unique fixed point of the induced $\langle F_\# \rangle$-action on $\Teich(S_2)$ is of the form $(c_1, c_2, c_3, t_1, t_2, t_3) \in \Teich(S_2)$, where
\begin{eqnarray*}
&&c_1 = c_2=2\operatorname{arcosh}\left(\frac{2+\sqrt{5}}{2}\right), \ \ \ \ \ \  c_3 = 2\operatorname{arcsinh}\left(\sqrt{\frac{5+3\sqrt{5}}{2}}\right),\\
&&t_1= t_2= 2\operatorname{arcosh}\left(\frac{1}{4}\sqrt{25+9\sqrt{5}}\right), \ \text{ and } \ \ \ t_3=-\operatorname{arcosh}\left(\frac{3+\sqrt{5}}{2}\right).\\
\end{eqnarray*}
\end{cor}

\begin{proof} 
We will apply Algorithm~\ref{algo:main} to establish our assertion.
\begin{enumerate}[\textit{Step} 1:]
\item We start with the regular hyperbolic $10$-gon $\mathcal{P}_F$, with interior angles equal to $\frac{2\pi}{5}$ and opposite sides identified (see Figure \ref{Fig:order10}(B)). Using Theorem~\ref{thm:geom_real}, $F$ can be realized as a $\frac{2\pi}{10}$ rotation around the center of $\mathcal{P}_F$. 

\begin{figure}[h]
  \centering
  \begin{subfigure}{0.45\textwidth}
    \centering
    \includegraphics[width=\textwidth]{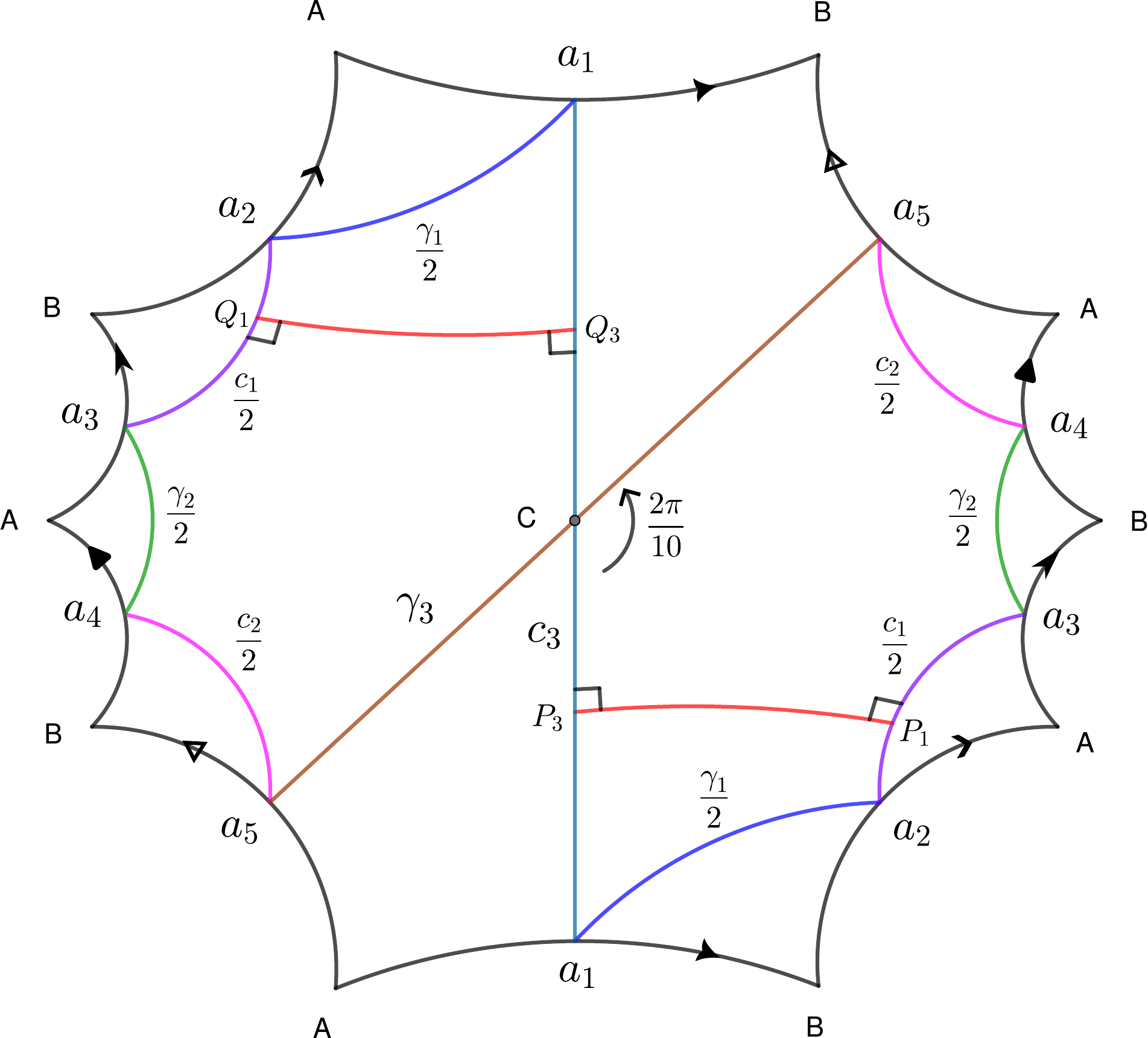}
    \caption{$10$-order induced action.}
    \label{fig: irregular hyp 10-gon}
  \end{subfigure}
  \hfill
  \begin{subfigure}{0.45\textwidth}
    \centering
    \includegraphics[width=\textwidth]{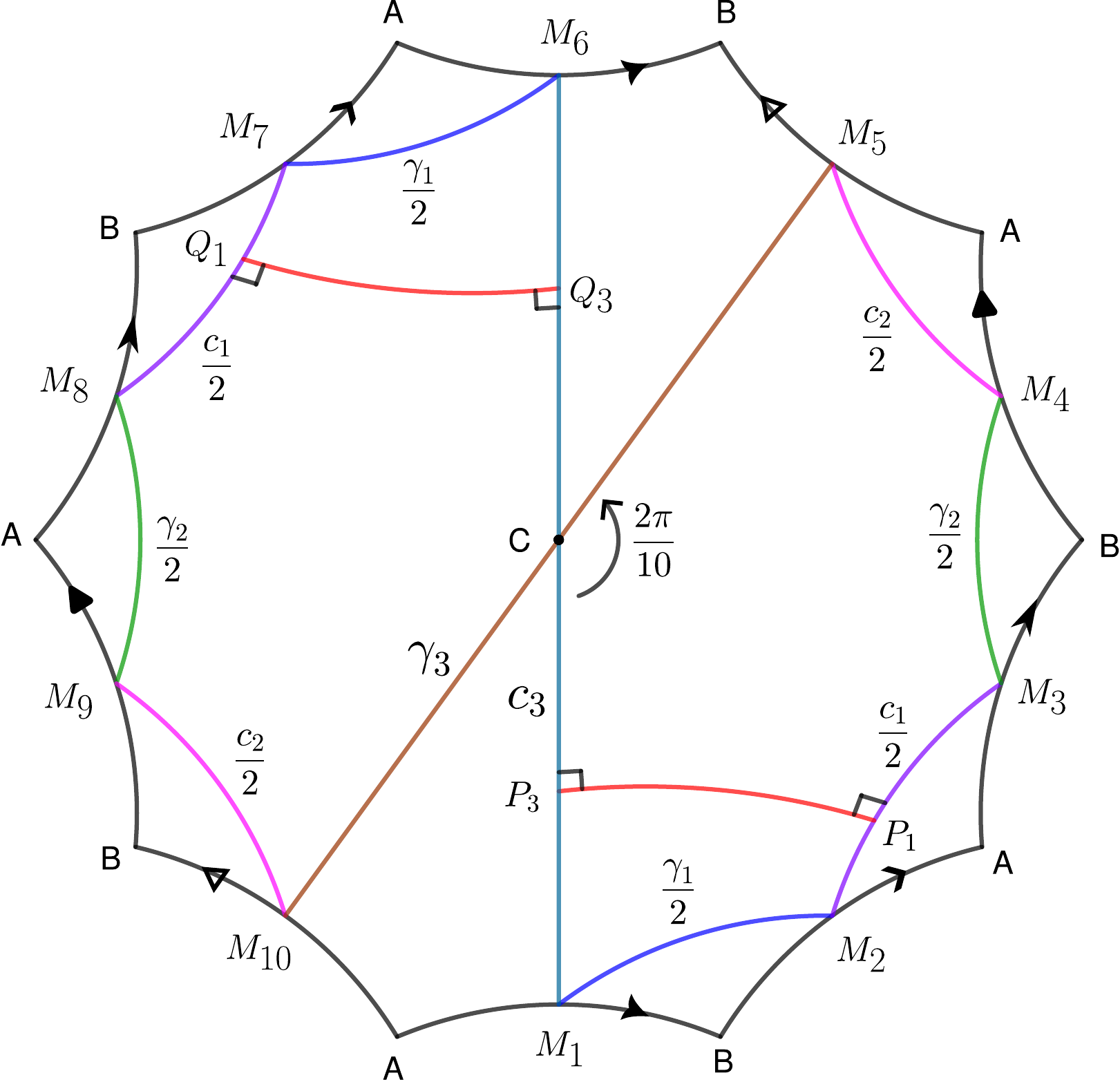}
    \caption{Fixed point of order $10$ action.}
    \label{fig: regular hyp 10-gon}
  \end{subfigure}
  \caption{Irreducible Type 1 $10$-order action on $\Teich(S_2)$.}
  \label{Fig:order10}
\end{figure}
\item We choose a pants decomposition of $S_2$  as follows: Consider the midpoints 
$M_1, \dots, M_{10}$ of the sides of $\mathcal{P}_F$ (in a counter-clockwise 
direction, in Figure \ref{Fig:order10}(B)). It is easy to see that the 
multicurve $\mathcal{P}=\{\gamma_1, \gamma_2, \gamma_3\}$ where $\gamma_1, \gamma_2, \gamma_3$ are the simple closed curves given by $M_1M_2M_7M_6$, $M_3M_4M_9M_8$, and $M_{10}M_5$ respectively and $\gamma_2=F^2(\gamma_1)$, forms a desired pants decomposition of $S_2$ in line with Algorithm \ref{algo:main}.
\item For computing the length and twist parameters of these pants curves, we first note that the simple closed curve $c_i$ is the image of $\gamma_i$ for $1\leq i \leq 3$, under the isometry $F$ as shown in Figure \ref{Fig:order10}(B).  It can also be shown that each $\gamma_i$ (and hence, its isometric image $c_i$) is the geodesic representative in its homotopy class. For a justification of this, we refer to Example \ref{exmp:induced action}. If there is no ambiguity, let $\gamma_i$ and $c_i$ continue to denote the hyperbolic lengths of the corresponding curves. Assuming $\mathcal{P}_F$ to be isometrically embedded in the Poincar\'e disc with its center coinciding with the center of the disc, the length and twist parameters (using hyperbolic trigonometry), are given below:

\textbf{Length parameters:} The hyperbolic side-length of $\mathcal{P}_F$ is  $$a=\operatorname{arcosh}\left(\frac{\cos^2 {\frac{\pi}{5}} + \cos \frac{\pi}{5}}{\sin^2 \frac{\pi}{5}}\right) = \operatorname{arcosh}\left(2+\sqrt{5}\right).$$ Since the hyperbolic length $CM_1=s$ (say), coincides with the \textit{inradius} of $\mathcal{P}_F$, we obtain (by \cite[Proposition 5.1]{Realization2})
$s=\operatorname{arcsinh}\left(\sqrt{\frac{5+3\sqrt{5}}{2}}\right).$ 
Thus, $\gamma_3=c_3=2s=2\operatorname{arcsinh}\left(\sqrt{\frac{5+3\sqrt{5}}{2}}\right)$, and for $i=1, 2$, $\gamma_i=c_i=2\operatorname{arcosh}\left(\cosh^2 {\frac{a}{2}} - \sinh^2 {\frac{a}{2}}\cos{\frac{2\pi}{5}}\right)=2\operatorname{arcosh}\left(\frac{2+\sqrt{5}}{2}\right).$

\textbf{Twist parameters:} For the pants decomposition $\{c_1,c_2,c_3\}$, consider the collection $\{\gamma_1,\gamma_2,\gamma_2\}$ as the seams curves. Due to the rotational symmetry of $\mathcal{P}_F$, the twist parameters are given by:  $|t_i|=P_1M_3+M_8Q_1=2P_1M_3, \ \text{ for } \ i=1, 2$ and $|t_3|=P_3Q_3=2CP_3$. Setting $R=\tanh \frac{s}{2}=\sqrt[4]{\frac{1}{5}},$ the midpoints $M_1, M_2, M_3$ are (in polar coordinates) 
$$M_1=Re^{-\iota \frac{\pi}{2}}, M_2=Re^{-\iota \frac{3\pi}{10}},M_3=Re^{-\iota \frac{\pi}{10}}.$$ 
The hyperbolic line passing through $M_2$ and $M_3$ is the circle: $(x-a)^2+(y-b)^2=r^2$, where 
$a=\frac{1+R^2}{R}{\sin \frac{\pi}{10} \cot \frac{\pi}{5}}=\sqrt{\frac{2+\sqrt{5}}{5}}, b=-\frac{1+R^2}{R} \sin \frac{\pi}{10}=-R,$ and $r=\sqrt{a^2+b^2-1}.$ Likewise, the hyperbolic line through $P_1$ and $P_3$ is the circle:
$x^2+(y-\frac{1}{b})^2=\frac{1}{b^2} - 1$. Thus, we obtain  
$$P_3=\iota (\frac{1}{b}+\sqrt{\frac{1}{b^2} - 1})=-\iota (\sqrt[4]{5} - \sqrt{\sqrt{5}-1})\ \text{ and } P_1=x_1 + \iota y_1$$ 
with 
$x_1=\frac{30 \sqrt{2 + \sqrt{5}} - 10 \sqrt{5 (2 + \sqrt{5})} - 4\sqrt{-925 + 415 \sqrt{5}}}{-75 + 43 \sqrt{5}} \ $  and 
 
$y_1=-\sqrt[4]{5} + \frac{2 \sqrt{10 (-803 + 362 \sqrt{5} + 6 \sqrt{225 - 95 \sqrt{5}} - 10 \sqrt{45 - 19 \sqrt{5}})}}{-75 + 43 \sqrt{5}}.$

Following our sign convention, we have 
$$t_3=-(2CP_3)=-2\ln|\frac{1+|CP_3|}{1-|CP_3|}|=-\ln{\frac{3 + \sqrt{5} + \sqrt{10 + 6 \sqrt{5}}}{2}}=-\operatorname{arcosh}\left(\frac{3+\sqrt{5}}{2}\right)$$ 
(here $|CP_3|$ denotes the Euclidean distance and $CP_3$ the hyperbolic distance).  Finally, using an well known formula for the hyperbolic distance in the Poincar\'e disc model, for $i=1, 2$, we get 
\begin{eqnarray*}
t_i&=&2P_1M_3=2\operatorname{arcosh}\left(1 + 2\frac{|P_1M_3|^2}{(1-x_1^2-y_1^2)(1-R^2)}\right)\\
& =& 2\operatorname{arcosh}\left(\frac{1}{4}\sqrt{25+9\sqrt{5}}\right) = \ln[\ \frac{1}{8}(17+9\sqrt{5}+3\sqrt{70+34\sqrt{5}})].
\end{eqnarray*}
\end{enumerate}
\end{proof}
\noindent It may be noted some of the computations preceding proof were done with the help of Mathematica~\cite{MATH24}.

\begin{cor}\label{cor:order8}
Let $F^{\prime} \in \Mod(S_2)$ be the Type 1 irreducible action given by $D_{F^{\prime}}=(8,0;(1, 2),(3, 8),(1, 8))$.  The Fenchel-Nielsen coordinates of the  unique fixed point of the induced $\langle F^{\prime}_\# \rangle$-action in $\Teich(S_2)$ is of the form $(\gamma_1, \gamma_2, \gamma_3, t_1, t_2, t_3)\in \Teich(S_2)$, where
\begin{eqnarray*}
&& \gamma_1=\gamma_2=t_1=t_2=2\operatorname{arcosh}\left(1+\sqrt{2}\right),\\ 
&&\gamma_3=2\operatorname{arsinh}\left(2 \sqrt{4+3\sqrt{2}} \right) \ \text{ and }t_3=-2\operatorname{arcosh}\left(\sqrt{2+\sqrt{2}}\right).
\end{eqnarray*}
\end{cor}

\begin{proof}
We will apply Algorithm~\ref{algo:main} to prove our assertion.
\begin{enumerate}[\textit{Step} 1:]
\item As in the proof of Corollary \ref{cor:order 10}, we realize $F^{\prime}$ as the $\frac{2\pi}{8}$ rotation of the regular hyperbolic $8$-gon $\mathcal{P}_{F^{\prime}}$ with opposite sides identified as shown in Figure \ref{fig:pants_c8action} and interior angles $2\pi/8$. 

\begin{figure}[h]
  \centering
    \includegraphics[width=40ex]{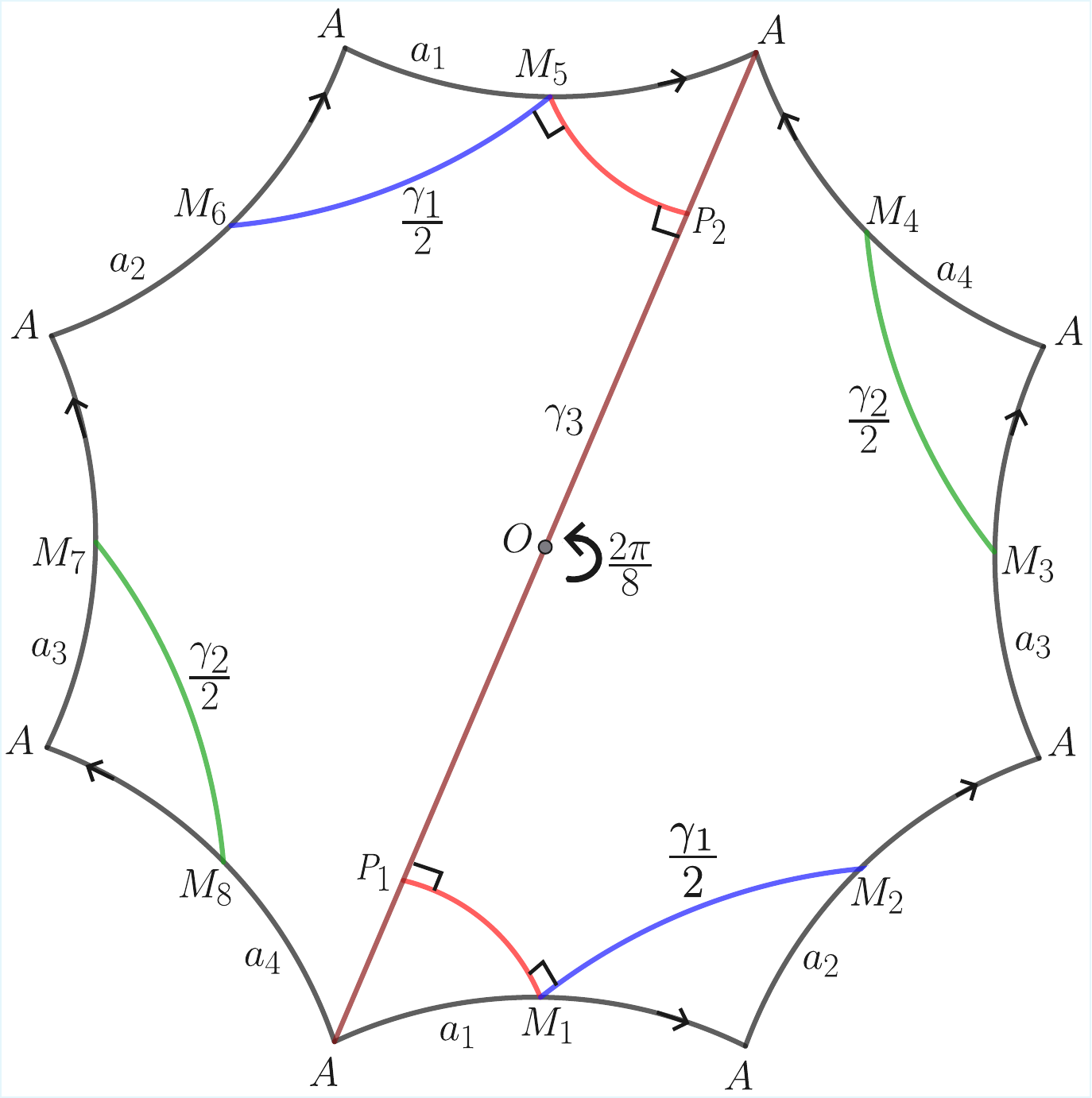}
    \caption{Irreducible Type 1 8-order action on $\Teich(S_2)$.}
    \label{fig:pants_c8action}
\end{figure}
\item Following Algorithm \ref{algo:main}, we choose a suitable pants decomposition $\{\gamma_1,\gamma_2 = F'^2(\gamma_1),\gamma_3\}$ of $S_2$ as indicated in Figure \ref{fig:pants_c8action} with respective isometric images $c_1,c_2,c_3$ under $F^{\prime}$.
\item To compute the length and twist parameters of $\gamma_1,\gamma_2,\gamma_3$, as in Corollary \ref{cor:order 10}, $\mathcal{P}_{F'}$ is viewed in the Poincar\'e disc with its center identified with the center of the disc.

\textbf{Length parameters:} The side-length of $\mathcal{P}_{F^{\prime}}$, is given by
$$a=AM_1A=\operatorname{arcosh}\left(\frac{\cos^2 {\frac{\pi}{8}} + \cos \frac{2\pi}{8}}{\sin^2 \frac{\pi}{8}}\right) = \operatorname{arcosh}\left(5+4\sqrt{2}\right).$$ As the \textit{inradius} of $\mathcal{P}_{F^{\prime}}$ is given by the hyperbolic length $OM_1=s$ (say), (using \cite[Proposition 5.1]{Realization2},) we have  $s=\operatorname{arsinh}\left(\sqrt{2+2\sqrt{2}}\right)$. 
From the right angled hyperbolic triangle $\bigtriangleup OAM_1$, we obtain the length  
$$\gamma_3=c_3=2 OA=2\operatorname{arsinh}\left(\frac{\sinh{s}}{\sin{\frac{\pi}{8}}}\right)=2\ln[\ 3+2\sqrt{2}+2\sqrt{4+3\sqrt{2}}].$$ For $i=1, 2$, we have 
$$ \ \ \ \ \gamma_i=c_i=2 M_1M_2=2\operatorname{arcosh}\left(\cosh^2 {\frac{a}{2}} - \sinh^2 {\frac{a}{2}}\cos{\frac{2\pi}{8}}\right)=2\ln[\ 1+\sqrt{2}+\sqrt{2+2\sqrt{2}}].$$ 

\textbf{Twist parameters.} For the pants decomposition $\{c_1,c_2,c_3\}$, consider the collection $\{\gamma_1,\gamma_2,\gamma_2\}$ as the seams curves. Using the rotational symmetry of $\mathcal{P}_{F^{\prime}}$ and our sign convention, we observe that $t_3=-(P_1P_2)=-(2 OP_1)$, and for $i=1, 2$, $t_i=+(Q_1M_2+M_6Q_2)=+(2Q_1M_2)$ where $Q_i$, $i=1,2$, (respectively, $P_i$, $i=1,2$,) denote the feet of the common perpendicular between the hyperbolic lines $\gamma_1$ and $\gamma_3$ lying on $\gamma_1$ ((respectively, on $\gamma_3$) as shown in Figure \ref{fig:pants_c8action}.  
It is straightforward to check that according to our sign convention, $t_1=t_2 >0$ and $t_3<0$. 

It can be checked that $Q_1$ and $Q_2$ coincide with the midpoints $M_1$ and $M_5$ respectively (see Figure \ref{fig:pants_c8action}): 
In fact, setting $R=\tanh{\frac{s}{2}}=\sqrt{\sqrt{2}-1}$, we have $M_1=(0, -R)$ and $M_2=(\frac{R}{\sqrt{2}}, -\frac{R}{\sqrt{2}})$. Thus, the geodesic $\gamma_1$ passing through $M_1$ and $M_2$ is given by the circle $(x-a)^2+(y-b)^2=r^2$, where $a=\sqrt{-\frac{1}{2} + \frac{1}{\sqrt{2}}}, b= \frac{a}{-\sqrt{2}+1} = -\sqrt{\frac{1}{2} + \frac{1}{\sqrt{2}}}$, and  $r=\sqrt{a^2+b^2-1}=R.$ 

The geodesic $\gamma_3$ is given by the line $y=\tan{\frac{3\pi}{8}} x = (\sqrt{2}+1) x$. Consequently, the common perpendicular between $\gamma_1$ and $\gamma_3$ passing through $P_1$ and $Q_1$, is given by $(x+a)^2+(y-b)^2=r^2$ where $a, b$ and $r$ are as above. Thus, the foot $Q_1$ of the common perpendicular on $\gamma_1$ is given by $Q_1=(0, -\sqrt{\sqrt{2}-1})=(0, -R)=M_1$. Similar arguments yield $Q_2=M_5$ implying that, for $i=1, 2$, 
$$t_i=2M_1M_2=\gamma_i.$$ 
Likewise, we have $P_1=(-\frac{1}{2} \sqrt{-6+5\sqrt{2}-2\sqrt{2(10-7\sqrt{2})}}, -\frac{1}{2} \sqrt{2+3\sqrt{2}-2\sqrt{2(2+\sqrt{2})}})$. Finally, using \cite[Proposition 4.3]{Anderson}, we have 
$$t_3=-2(OP_1)=-2\operatorname{arcosh}{\sqrt{2+\sqrt{2}}}=-2\ln[\sqrt{1+\sqrt{2}} + \sqrt{2+\sqrt{2}}].$$ This completes the proof.
\end{enumerate}
\end{proof}
We conclude this subsection with an immediate consequence of Corollary \ref{Cor:order 10} describing the fixed point of a higher power of the irreducible Type 1 action of order $10$ on $S_2$ mentioned in the corollary.
\begin{exmp}
Let $\tilde{F}$ be the irreducible Type 1 action on $S_2$ given by $D_{\tilde{F}}=(5,0 ;(2,5),(2,5),(1,5))$. It can be checked that $\tilde{F}=F^2$ where $F$ is the order $10$ action on $S_2$ mentioned in Corollary \ref{Cor:order 10}. In fact, $\tilde{F}$ can be realized as a $\frac{2 \pi}{5}$ rotation of the same polygon $\mathcal{P}_F$. Moreover, as $\tilde{F}$ be the irreducible Type 1 action,  its induced action on $\Teich\left(S_2\right)$ has a unique fixed point (\cite[Proposition 4.1]{Realization2}), it follows that $\tilde{F}_{\#}$ and $F_{\#}$ have the same fixed point in $\Teich\left(S_2\right)$, (whose Fenchel-Nielsen coordinates) described in Corollary \ref{Cor:order 10}.
\end{exmp}

\subsection{Branch loci of $(r,s)$ compatible pairs of Type 1 irreducible actions}\label{sec.4} 
\begin{algo}\label{algo:compatible_pair} Given a Type 1 irreducible $F \in \operatorname{Mod}\left(S_g\right)$ of order $n$, we consider the compatible pair $G:=(F, F^{-1})$ realized via compatibility along a common fixed point. Then an algorithm to describe the Fenchel-Nielsen coordinates of the fixed points of the induced $\langle G_{\#}\rangle$ action on $\Teich(S_{2 g})$ is outlined below:
\begin{enumerate}[\textit{Step} 1:]
\item Start with the unique hyperbolic polygon $\mathcal{P}_F$ realizing $F$ as an isometry as mentioned in Theorem \ref{thm:geom_real}.
\item If $\mathcal{P}_F$ has $k$ sides, obtain a hyperbolic $2k$-gon $\mathcal{P}_G$ representing a hyperbolic structure on $S_{2g}$ that realizes $G$ as an isometry, as follows:
\begin{enumerate}[\textit{Step} 2(a)]
\item Construct a hyperbolic $(k+1)$-gon $\mathcal{P}^{\prime}$, with $k$ sides of equal length such that the corresponding $(k-1)$ interior angles bounded by those $k$ sides are equal. The remaining side, distinguished as $\ell$, is not necessarily equal to the other $k$ sides. 
\item In order that $\mathcal{P}^{\prime}$ realizes a hyperbolic structure on $S_{g}^1$ with $F$ as an isometry, it follows from basic hyperbolic trigonometry that the remaining two interior angles of $\mathcal{P}^{\prime}$ must be equal to each other (not necessary equal to the $k-1$ interior angles mentioned above). 
\item On these $k$ equal sides apply the side-pairing relations as in the polygon $\mathcal{P}_F$. Note such a polygon $\mathcal{P}^{\prime}$ always exist representing $F$ as an isometric action on the hyperbolic surface $S_{g}^1$.
\item Construct another hyperbolic ($k+1$)-gon $\mathcal{P}^{\prime \prime}$ by reflecting $\mathcal{P}^{\prime}$ along its distinguished side $\ell$. $\mathcal{P}^{\prime \prime}$ then has the same side-pairing relations as of $\mathcal{P}^{\prime}$.
\item Obtain the desired hyperbolic $2k$-gon $\mathcal{P}_G$ by combining $\mathcal{P}^{\prime}$ and $\mathcal{P}^{\prime \prime}$ identified along the common distinguished side $\ell$. Note that by construction, all sides of $\mathcal{P}_G$ has equal length, $(2k-2)$ interior angles are equal and the remaining two interior angles are also equal to each other (but not necessarily to the $(2k-2)$ interior angles).  
\end{enumerate} 
\item Construct a suitable pants decomposition on $\mathcal{P}_G$ in the following way:
\begin{enumerate}[\textit{Step} 3(a)]
\item Pick the common distinguished side of $\mathcal{P}^{\prime}$ and $\mathcal{P}^{\prime \prime}$ as the first pants curve. Call it $\gamma$.
\item  Choose a homotopically non-trivial simple closed curve, possibly along the boundary of the polygon $\mathcal{P}^{\prime}$, disjoint from $\gamma$ up to homotopy, and consider its orbit under the action.
\item If the number of homotopically disjoint curves in that orbit is $3 g-2$, consider their counterparts from $\mathcal{P}^{\prime \prime}$. These $6g-4$ curves together with $\gamma$, form a pants decomposition for $S_{2 g}$.
\item Otherwise, find a suitable nontrivial simple closed curve in $\mathcal{P}^{\prime}$, lying in the complement of the previous orbit and $\gamma$ (such a curve always exists), and consider its orbit under the action.
\item Repeat the steps 3(b) - 3(d) until the total number of homotopically disjoint simple closed curves obtained from distinct orbits in the polygon $\mathcal{P}^{\prime}$ (other than $\gamma$) adds up to $3 g-2$.
\end{enumerate}
\item Viewing $\mathcal{P}_G$ inside the Poincaré disc with proper identifications, compute the length and twist parameters associated with the corresponding chosen pants curves thereby obtaining the Fenchel-Nielsen coordinates of the hyperbolic structure with the help of hyperbolic trigonometry.
\end{enumerate}
\end{algo}
Applying Algorithm \ref{algo:compatible_pair}, we describe the fixed points of a cyclic action of order $4$ on $\Teich(S_2)$ as follows.
\begin{cor}\label{Cor:Type2}
Let $G \in \Mod(S_2)$ be the Type 2 action of order $4$ with $D_G = (4, 0; (1, 2)^{[2]}, (1, 4), (3, 4))$ realized as a $(3,3)$-compatible pair $(F,F^{-1})$, where $D_F = (4,0;(1,2),(1,4),(1,4))$. Then the fixed points of the induced $\langle G_{\#} \rangle$-action on $\Teich(S_2)  \approx \mathbb{R}^6$ is a two dimensional submanifold whose Fenchel-Nielsen coordinates are given by $\left\{\left(\gamma_1, \gamma_2, \gamma_1, t, 0,-t\right)\right\}$, where 
\begin{eqnarray*}
&& \gamma_1=2 \operatorname{arcosh}\left(\cosh \frac{s}{2} \sin \alpha\right),  \gamma_2=\operatorname{arcosh}\left(\cosh ^4 s-2 \cosh ^3 s+2 \cosh s\right),\\
&& t= \frac{\gamma_1}{2}-\operatorname{arcoth}\left(\sinh (s-x) \tan \alpha\right), \ \text{and }x=\operatorname{artanh}\left(\operatorname{coth} \frac{s}{2} \cos \alpha\right)
\end{eqnarray*}
with $s>\operatorname{arcosh}\left(\cot ^2 \frac{\alpha}{2}\right)$, and $0< \alpha<\frac{\pi}{3}$.
\end{cor}

\begin{proof}  
\begin{enumerate}[\textit{Step} 1:]
\item It can be checked that the action $F$ of order $4$ on $S_1$ given by $D_F=\left(4,0 ;(1,2),(1,4)^{[2]}\right)$ can be realized as the $\pi/2$ rotation of a square. 
\item We construct a desired hyperbolic octagon $\mathcal{P}_G$ realizing $G$ as a hyperbolic isometry on $S_2$ (as mentioned in Algorithm \ref{algo:compatible_pair}), with the side-pairing relations indicated by the arrows in Figure \ref{Fig:order4_compatibility}(A). Here six interior angles are equal to $\alpha$ (say) and the remaining two interior angles equal to $2 \beta$. 

\begin{figure}[H]
  \centering
  \begin{subfigure}{0.43\textwidth}
    \centering
    \includegraphics[width=\textwidth]{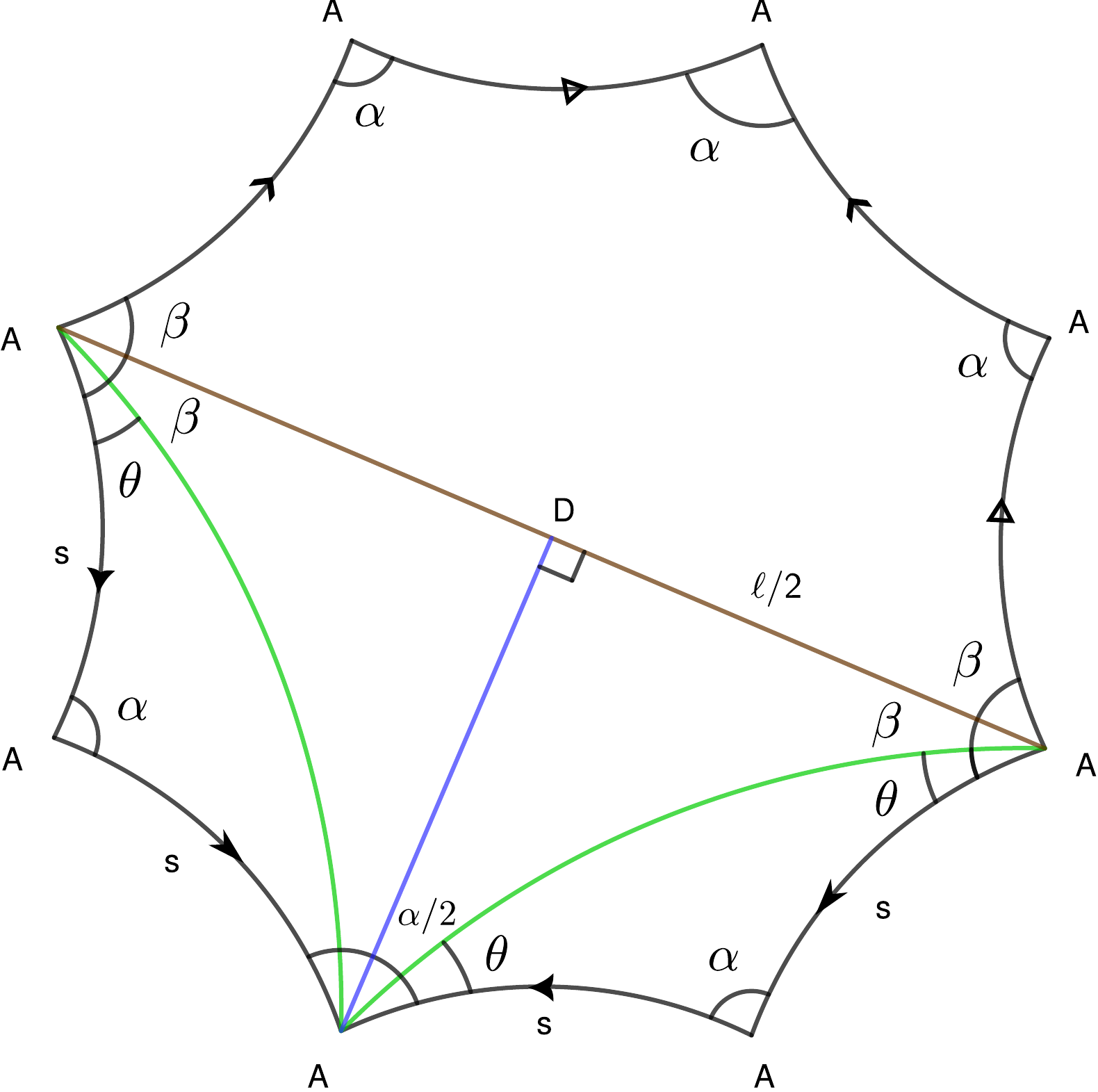}
    \caption{}
    \label{528gon}
  \end{subfigure}
  \hfill
  \begin{subfigure}{0.43\textwidth}
    \centering
    \includegraphics[width=\textwidth]{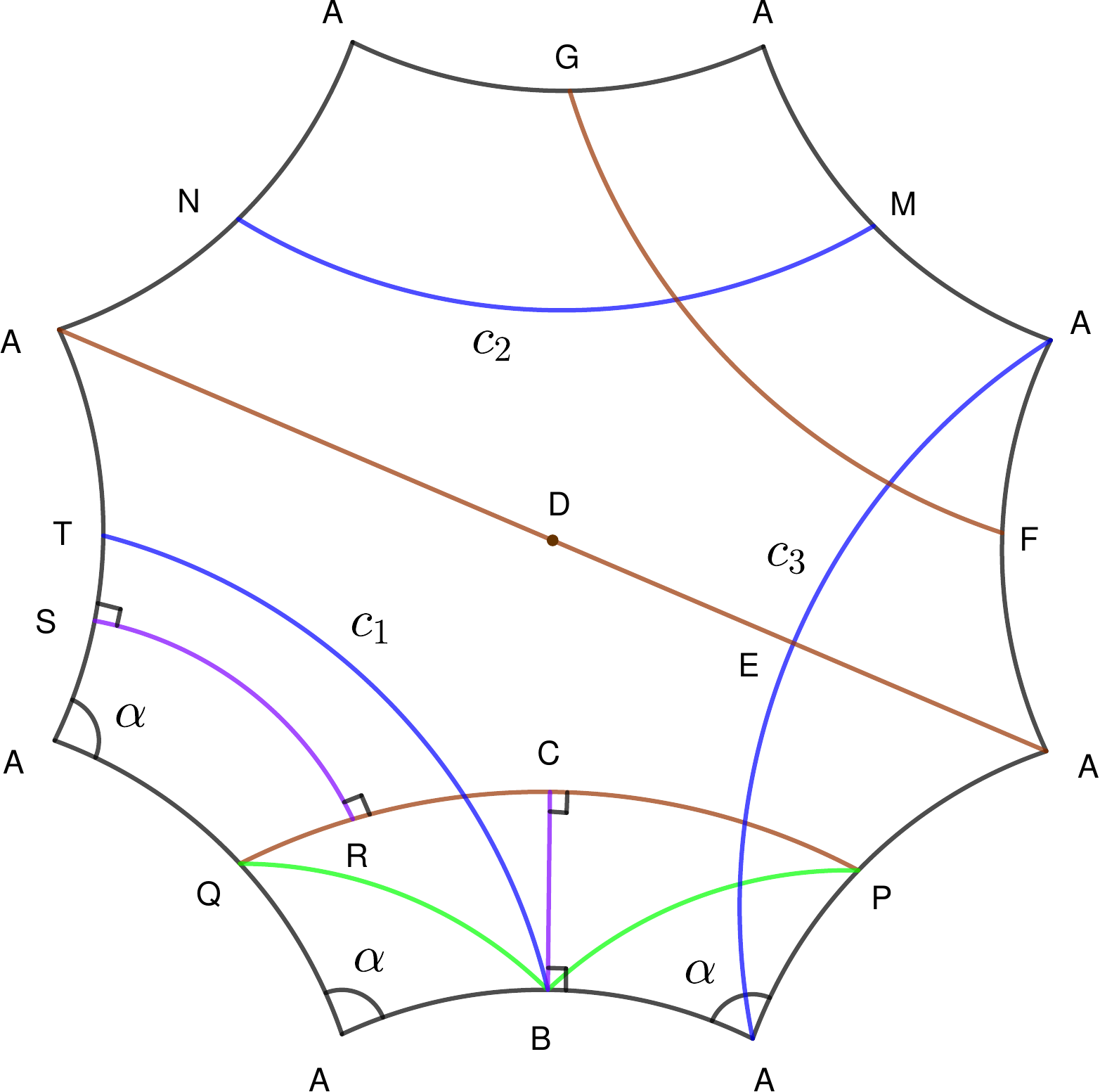}
    \caption{}
    \label{fig: L&T of 4 order compatibility}
  \end{subfigure}
  \caption{A Type 2 action of order $4$ on $S_2$.}
  \label{Fig:order4_compatibility}
\end{figure}
 
Using the Gauss-Bonnet Theorem and area formula of hyperbolic polygons, we obtain $3 \alpha+2 \beta=\pi$. It follows from basic hyperbolic trigonometry that such a hyperbolic octagon exists if and only if $\alpha \in(0, \pi / 3)$ and $\beta=\pi / 2-3 \alpha / 2$. Note that as mentioned in Step 2(e) of Algorithm \ref{algo:compatible_pair}, all sides of $\mathcal{P}_G$ have equal 
length, say, $s$. In fact, $\mathcal{P}_G$ is constructed from two identical pieces of hyperbolic pentagons $\mathcal{P}^{\prime}$ and $\mathcal{P}^{\prime \prime}$ identified across the line segment $A D A$ (see Figure \ref{Fig:order4_compatibility}(A)). Each of the pentagons represents a copy of the compact surface $S_{1,1}$ such that after capping off its boundary (or equivalently, collapsing the side $A D A$ of the pentagon to a point), it yields a copy of $S_1$ and the restriction of $G$ onto each pentagon resembles to the (isometric) actions $F$ and $F^{-1}$ respectively. Thus four sides (excluding the $A D A$) of each pentagon must have the same length.

Thus, it suffices to describe the action of $G$ on the pentagon $\mathcal{P}^{\prime}$. Using hyperbolic trigonometry, it follows that the line segment $AD$ divides $\mathcal{P}^{\prime}$ into two congruent quadrilaterals (see Figure \ref{Fig:order4_compatibility}(A)). Thus, $A D$ bisects the angle $A$ and perpendicularly bisects the $A D A$ at $D$. The side length $d$ of either isosceles triangle opposite to the vertex $A$ is given by $
\cosh d=\cosh ^2 s-\sinh ^2 s \cos \alpha$ and using using \cite[Theorem 2.2.1]{PBuserBook}, we have 
$\cosh s=\frac{\cos \theta+\cos \theta \cos \alpha}{\sin \theta \sin \alpha}=\cot \theta \cot \frac{\alpha}{2}.$ As $0<\theta<\alpha / 2<\pi / 6$, from the above identities, it follows that $\cosh s >\cot ^2 \frac{\alpha}{2}$ and hence, $s>\operatorname{arcosh}\left(\cot ^2 \frac{\pi}{6}\right)=\ln [3+2 \sqrt{2}]$.  

\item An easy way to obtain pants decomposition for $S_2$, using $\mathcal{P}_G$, is the multicurve ~$\{A B A, A D A, A M A\}$ as shown in Figure \ref{Fig:order4_compatibility}(B). It now remains to compute the desires length and twist parameters of the same. 

\textbf{Length Parameters:} We first note that the common side $A D A$ defines a smooth geodesic on $S_2$ called $\gamma_2$ (say), as its angles of intersection at both ends are equal. Denoting the length of $ADA$ by $\ell$, we see that $\tanh \frac{\ell}{2}=\tanh d \cos (\beta-\theta)$ (\cite[Theorem 2.2.1]{PBuserBook}). Substituting $d, \theta, \beta$ in above equations, we have
$$
\cosh \gamma_2=\cosh \ell=\cosh ^4 s-2 \cosh ^3 s+2 \cosh s
$$ 
Let the geodesic representatives of $A B A$ and $A M A$ be $\gamma_1$ and $\gamma_3$, respectively. A typical curve homotopic to $A B A$ would be of the form $P Q$ as in Figure \ref{Fig:order4_compatibility}(B) where $A P=A Q=x$ for some $0<x<s$. For $P Q$ to be the geodesic $\gamma_1$ on $S_2$, we must have $\angle A P Q=\pi-\angle A Q P=\phi$ (say), in quadrilateral $\square A B A P C Q$. It can be checked that $\phi=\frac{\pi}{2}$, and hence $\angle A P Q=\angle A Q P=\frac{\pi}{2}$. Likewise, $B C$ is the perpendicular bisector of $P Q$ and using (\cite[Theorem 2.3.1]{PBuserBook}), we observe
$$
\cosh \frac{\gamma_1}{2}=\cosh P C=\cosh \frac{s}{2} \sin \alpha, \text { and } \tanh x=\operatorname{coth} \frac{s}{2} \cos \alpha.
$$
From the symmetry of the polygon, we have $\gamma_3=\gamma_1$.

\textbf{Twist Parameters:} Consider the collection $\left\{c_1, c_2, c_3\right\}$ of seams curves given by $c_1=B T, c_2=A E A, c_3=M N$ where $B, T, M, N$ are the midpoints of their respective sides, and $c_3$ the image of $c_1$ under the reflection by $A D A$. Let $t_i$ be the twist parameters of $\gamma_i$ for $i=1,2,3$. From the symmetry, it immediately follows that $t_2=0$. Denoting $t_1=t$, consider the two common perpendiculars of $\gamma_1(=P Q)$, passing through $A B A \sim A T A$, along the seam curve $c_1$. The perpendiculars $B C$ and $S R$ (see Figure \ref{Fig:order4_compatibility}(B)) both will either be on the left of $A Q$ lying within the polygon, or on the right side of $A Q$ means completely outside the polygon. The latter is ruled out by basic hyperbolic trigonometry. Therefore, for the left case, using \cite[Theorem 2.3.1]{PBuserBook}), we get
$$
\operatorname{coth} Q R=\sinh (s-x) \tan \alpha, \text { and hence }t=+(\frac{\gamma_1}{2}-Q R) .
$$
From the symmetry of $\mathcal{P}_G$ and the choice of the pants and seams curves, it follows that $t_3=-t_1=-t$ completing the proof.
\end{enumerate} 
\end{proof}

\section{Concluding remarks: Actions induced by Type 1 irreducibles on \T spaces }\label{sec:conclusion}
It is worth noting that Algorithm \ref{algo:main} can also be utilised to obtain an explicit description of the Teichm\"uller isometries induced by irreducible Type 1 actions of $S_g$ ($g\geq 2$) via Fenchel-Nielsen coordinates. This is certainly an important aspect of the algorithm, as describing induced actions of irreducible Type 1 mapping classes on the \T space has always been quite hard due to the non-existence of multicurves that are preserved under such actions.
\begin{algo}[A sketch]\label{algo:induced action} Given a Type $1$ irreducible $F \in \Mod(S_g)$ of order $n$, its induced action on $\Teich(S_g)$ is described as follows: 
\begin{enumerate}[\textit{Step} 1:]
\item Given an arbitrary $[\xi] \in \Teich(S_g)$ where $\xi \in \operatorname{HypMet}\left(S_g\right)$, start with a suitable hyperbolic polygon $P$ (not regular in general) in the upper half plane (or equivalently, in the Poincar\'e disc) with the side-pairing as $\mathcal{P}_F$ representing the hyperbolic surface $(S_2,\xi)$. Then $F$ acts as a differemorphism (not an isometry in general) on $P$ of order $n$.  
\item Construct a pants decomposition of $S_g$ following a similar method as in \textit{Step 2} of Algorithm \ref{algo:main} so that the pants curves represent the unique hyperbolic geodesics in their respective homotopy classes. 

\item Viewing $P$ as a hyperbolic polygon in the upper-half plane $\mathbb{H}^2$ or the Poincar\'e disc $\mathbb{D}$ with proper identifications, compute the length and twist parameters associated with the corresponding pants curves and their images, thereby obtaining the Fenchel-Nielsen coordinates of $[\xi]$ and $F_{\#}([\xi])$  given by $P$ using basic hyperbolic trigonometry.
\end{enumerate}
\end{algo}
To avoid computational complexities, we demonstrate Algorithm \ref{algo:induced action} through the irreducible Type 1 action of order $10$ on $S_2$ discussed in Corollary \ref{Cor:order 10} and describe its induced action on $\Teich(S_2)$.
\begin{exmp} \label{exmp:induced action}
Let $F \in \Mod(S_2)$ be the Type 1 irreducible action with $D_F=(10, 0; (1, 2), (2, 5), (1, 10))$. Given an arbitrary $[\xi] \in \Teich(S_2)$, where $\xi \in \operatorname{HypMet}\left(S_2\right)$, it follows that the hyperbolic surface $(S_2,\xi)$ can be represented by a suitable hyperbolic $10$-gon $P$ (not regular in general) with opposite sides identified as shown in Figure \ref{Fig:order10}(A). Let's denote the hyperbolic side-lengths of $P$ by $a_1,\ldots,a_5$ with respective midpoints $M_1, \dots, M_{10}$ arranged in a counter-clockwise order (see Figure \ref{Fig:order10}(A)). We choose the pants decomposition $\mathcal{P}=\{\gamma_1, \gamma_2, \gamma_3\}$ (similar to the construction in Corollary \ref{cor:order10_fixedpt}) with its three pant curves given by $M_1M_2M_7M_6$, $M_3M_4M_9M_8$, and $M_{10}M_5$ respectively. We first show that $\gamma_i$ and $c_i:=F(\gamma_i)$, for $i=1, 2, 3$, are actually the geodesic representatives in their respective homotopy classes. Then we indicate how to explicitly compute their length and twist parameters using hyperbolic trigonometry and the structure of $P$.

\textbf{Length Parameters:}
In the upper-half plane $\mathbb{H}^2$, identifying the edge $a_1$ with a hyperbolic line segment on the unit circle such that $M_1=\iota=(0,1)$,  we consider a one-parameter family $\{\Gamma_l\}$ of curves homotopic to $c_3$, given by the line segments joining $A_l=(\tanh l, \operatorname{sech} l)$ and $B_l=(e^d\tanh l, e^d\operatorname{sech} l)$, where $d=d_{\mathbb{H}^2}(M_1, M_6) \text{ and }l \in [-\frac{a_1}{2}, \frac{a_1}{2}]$ with $|l|=d_{\mathbb{H}^2}(\iota, A_l)$. Using \cite[Theorem 1.2.6]{Katok} and the convention $\Gamma_l:=\text{length}(\Gamma_l)=d_{\mathbb{H}^2}(A_l, B_l)$, we obtain

$$\cosh {\frac{\Gamma_l}{2}} = \frac{|A_l-\overline{B_l}|}{2\{\text{Im}(A_l)\text{Im}(B_l)\}^{\frac{1}{2}}} = \sqrt{\cosh^2 \frac{d}{2} + \sinh^2l\sinh^2\frac{d}{2}}$$
   
It follows that, $\Gamma_l$ attains its minimum value only at $l=0$ and then, we have $A_0=M_1$, $B_0=M_6$. Thus $\Gamma_0=d$ and hence  $c_3=\text{length}[c_3]=\Gamma_0=d$. Using \cite[Theorem 2.2.6]{PBuserBook}, it follows that $\angle C$ ($C$ is the center of $P$) of the hyperbolic quadrilateral $\square CM_5BM_6$ is given by $\cos C = -\cosh \frac{a_1}{2}\cosh \frac{a_5}{2} \cos B + \sinh \frac{a_1}{2}\sinh \frac{a_5}{2}$.

Let $0\leq x\leq a_2, 0\leq y\leq a_3$ are two real numbers, and $P_x, Q_y, R_y, T_x$ are the points on $a_2, a_3, a_8\sim a_3, a_7\sim a_2$ respectively with $x=d_{\mathbb{H}^2}(P_x, A), y= d_{\mathbb{H}^2}(A, Q_y), a_3-y=d_{\mathbb{H}^2}(R_y, B)$, and $a_2-x=d_{\mathbb{H}^2}(B, T_x)$. Now, considering the two-parameter family of curves $\{\Gamma_{x, y}=P_xQ_yR_yT_x\}$ homotopic to the curve $c_1$, and the hyperbolic triangles $\bigtriangleup P_xAQ_y$ and $\bigtriangleup R_yBT_x$ with $P_xQ_y$ and $R_yT_x$ as one of their sides respectively. It turns out (using Mathematica) that $\Gamma_{x, y}:=\text{length}(\Gamma_{x, y})=P_xQ_y+R_yT_x$ attains its minimum at $(x,y)=(\frac{a_2}{2},\frac{a_3}{2})$ and we obtain 
$$c_1=2\operatorname{arcosh}\left(\cosh \frac{a_2}{2} \cosh \frac{a_3}{2} - \sinh \frac{a_2}{2} \sinh \frac{a_3}{2} \cos \alpha\right),$$
where $\alpha$ is the angle between $a_2$ and $a_3$. The computation for $c_2:=\text{length}[c_2]=\text{length}[F(\gamma_2)]$ is similar, and given by $c_2=2\operatorname{arcosh}\left(\cosh \frac{a_4}{2} \cosh \frac{a_5}{2} - \sinh \frac{a_4}{2} \sinh \frac{a_5}{2} \cos \beta\right)$, where $\beta$ is the angle between $a_4$ and $a_5$.
 
 \textbf{Twist Parameters:} To compute the twist parameter $t_i$ of $c_i=F(\gamma_i)$ for $i=1, 2, 3$, we take the collection $\{\gamma_1, \gamma_2, \gamma_3\}$ as our required seams curves. Denoting the common perpendiculars by $P_1P_3$ and $Q_1Q_3$ (see the Figure \ref{Fig:order10}(A)) between $c_1$ and $c_3$ along the seam curve $\gamma_1$, we have $t_1=P_1M_3 + M_8Q_1$.

 Assuming $a_1$ to be a segment of the unit circle as above, we obtain $B=(\tanh {\frac{a_1}{2}}, \operatorname{sech}{\frac{a_1}{2}})$. Likewise, the hyperbolic line segments $a_2, a_3$ and their respective midpoints $M_2, M_3$, can also be determined explicitly. This further leads to determining the common perpendicular between $c_3$ and $\frac{c_1}{2}$ (joining $M_2$ and $M_3$), the corresponding points of intersections $P_1,P_3$, and the mid point $M_3$ of $c_3$, and the hyperbolic length $P_1M_3$. Likewise, the length $M_8Q_1$ can also be determined, summing up the computation for $t_1$. The twist $t_2$ is given by $t_2=-(P_3Q_3)$. The computation for $t_3$ is similar to the method discussed for $t_1$. 
\end{exmp}

\section{Acknowledgements} The second author was supported by the UGC doctoral fellowship, Govt. of India. The third author was supported by an ANRF MATRICS grant, Govt. of India, File No. MTR/2023/000610. 

\bibliographystyle{plain}
\bibliography{branchloci_coods}
\end{document}